\documentclass[10pt,a4paper]{article}
\usepackage{amsmath}
\usepackage{amsfonts}
\usepackage{amssymb}
\usepackage{graphicx}
\usepackage{amsthm}
\usepackage[all]{xy}
\usepackage{lineno,hyperref}

\usepackage[subnum]{cases}

\newcounter{oftheorem}[subsection]
\newenvironment{mytheorem}[1]%
{\begin{trivlist}
     \renewcommand{\theoftheorem}{#1~\thesection.\arabic{oftheorem}}
     \refstepcounter{oftheorem}
     \item[\hspace{\labelsep}\bf\thesection.\arabic{oftheorem} #1]}%
{\end{trivlist}}
\newenvironment{definition}{\begin{mytheorem}{Definition}\it}{\end{mytheorem}}

\newenvironment{proposition}{\begin{mytheorem}{Proposition}\it}{\end{mytheorem}}
\newenvironment{theorem}{\begin{mytheorem}{Theorem}\it}{\end{mytheorem}}

\newenvironment{remark}{\begin{mytheorem}{Remark}}{\end{mytheorem}}
\newenvironment{lemma}{\begin{mytheorem}{Lemma}}{\end{mytheorem}}

\author{Stavros Anastassiou\\
Department of Mathematics\\
Aristotle University of Thessaloniki\\
GR-54124 Thessaloniki, Greece\\
sanastassiou@gmail.com}

\title{Singularities of $3$--d vector fields preserving the form of Martinet}
\begin{document}
\maketitle
\begin{abstract}
We study the local structure of vector fields on $\mathbb{R}^3$ which preserve the Martinet $1$-form $\alpha=(1+x)dy\pm zdz$. We present the classification of their singularities, up to diffeomorphisms preserving the form $\alpha$, as well as their transversal unfoldings. We are thus able to provide a fairly complete list of the bifurcations such vector fields undergo.   \end{abstract}
%===========================================================================
\section{Introduction}
The study of singularities of vector fields, along with the subsequent analysis of the bifurcations they undergo, is a subject that has attracted a lot of attention. In case the vector fields preserve a special structure (symplectic structure, volume form, symmetry...), it is only natural to study their classification up to transformations (hommeomorphisms/diffeomorphisms) which also preserve this structure. In this article, we engage ourselves with the local classification of vector fields on $\mathbb{R}^3$ which preserve the Martinet $1$-form.

To make things precise, let $M$ be a smooth (i.e. $C^{\infty}$), $3$--dimensional,  manifold and $w$ a differential $1$--form on $M$. At a generic point of $M$, the condition $w\wedge dw\neq 0$ is fulfilled and, locally, $w$ reduces to the Darboux normal form $dz+xdy$.

There may exist, however, a smooth submanifold $M_1$, of codimension $1$ in $M$, at a generic point $p$ of which the $1$--form w satisfies the following conditions:
$$(i)\ w|_p\neq 0,\ (ii)\ (dw)|_p\neq 0,\ (iii)\ ker (dw)|_p \pitchfork M_1.$$
At such a point, Martinet (\cite{Martinet}) showed that the $1$-form can be written (at least locally) as $\alpha =(1+x)dy\pm zdz$. We shall refer to $\alpha$ as ``the Martinet 1--form".

This is a $1$--form on $M=\mathbb{R}^3$, which fulfils the condition $\alpha \wedge da\neq 0$ everywhere, except the $M_1=\{(x,y,z)\in M,z=0\}$ surface. At the point $p=(0,0,0)$ of this surface, the three conditions above are indeed satisfied; that is, $\alpha |_p=dy\neq 0,\ (da)|_p=dx\wedge dy \neq 0$, while $ker (a|_p)=\langle \frac{\partial}{\partial x},\frac{\partial}{\partial z}\rangle$, which is transverse to $M_1$.

In dimension three, the Darboux and the Martinet models are the only stable models for 1--forms (the same holds for their generalizations, in any odd--dimensional manifold \cite{Zhitomirskii}). Vector fields and diffeomorphisms preserving the contact form (and more generally, the contact structure) were studied in \cite{Lychagin1,Lychagin2}. Here, we engage ourselves with the study of vector fields preserving the Martinet form.

To be more precise, let $\mathcal{X}(\mathbb{R}^3,\alpha)$ stand for the set of germs, at the origin, of those $C^{\infty}$ vector fields $X$ of $\mathbb{R}^3$, which preserve the $1$--form $\alpha$, that is, those vector fields which satisfy equation $\mathcal{L}_X\alpha =0$, where $\mathcal{L}_X$ stands for the Lie derivative in the direction of $X$. Let also $Diff(\mathbb{R}^3,\alpha)$ denote the set of germs, at the origin, of $C^{\infty}$ diffeomorphisms  preserving the form $\alpha$, that is, diffeomorphisms $\phi$ of $\mathbb{R}^3$ with $\phi ^*\alpha=\alpha$. In this paper, we wish to classify vector fields of $\mathcal{X}(\mathbb{R}^3,\alpha)$, up to diffeomorphisms belonging in $Diff(\mathbb{R}^3,\alpha)$.

Towards this end, in section $2$ we give the general form of vector fields and diffeomorphisms preserving the Martinet model and show that their study can be simplified, using the fact that they depend on univariate functions. In section $3$, we recall the local classification theorem, for 1--dimensional vector fields, and use it to give the local models of vector fields of interest here. In section 4, we present transversal unfoldings for the singularities found in the previous section and study the simplest bifurcation occurring; in an analogous way, one can study bifurcations of arbitrary, finite, codimension.

Throughout this article, we assume everything to be smooth, that is, we work in the $C^{\infty}$ category. Furthermore, our considerations are purely local. In what follows, even when, for the sake of simplicity, not explicitly stated, we study germs of functions and germs of vector fields at the origin. The neighbourhood of the origin, which interests us, does not include the $x=-1$ line, on which $\alpha$ descents to $\pm zdz$.      

For the interested reader, excellent references on the classification of vector fields should include \cite{Chow-Li-Wang, Zhitomirskii,Banyaga-Llave, Meiss-Dullin}; on singularity theory see \cite{Brocker}; some recent articles in the same spirit are \cite{Kourliouros1,Kourliouros2}, while this article is directly connected with \cite{Anastassiou}.
%===============================================================     
\section{The vector fields of interest}
Consider $\mathbb{R}^3$ equipped with the Martinet form $\alpha$. If $X\in \mathcal{X}(\mathbb{R}^3,\alpha)$, write $X=X_1\frac{\partial}{\partial x}+X_2\frac{\partial}{\partial y}+X_3\frac{\partial}{\partial z}$. Equation $\mathcal{L}_X\alpha=0$ becomes:
\[
\begin{cases}
(1+x)\frac{\partial X_2}{\partial x}\pm z\frac{\partial X_3}{\partial x}&=0\\
X_1+(1+x)\frac{\partial X_2}{\partial y}\pm z\frac{\partial X_3}{\partial y}&=0\\
(1+x)\frac{\partial X_2}{\partial z}\pm X_3\pm z\frac{\partial X_3}{\partial z}&=0
\end{cases}.
\]
By studying the system above, one finds that $X=-(1+x)X_2(y)\frac{\partial}{\partial x}+X_2(y)\frac{\partial}{\partial y}$. Note that its third component vanishes identically. We may thus restrict our attention to the $z=0$ plane and study $C^{\infty}$ vector fields defined on it, which preserve the $1$--form $\mu=(1+x)dy$ (that is, vector fields $X$ satisfying $\mathcal{L}_X\mu=0$). We shall denote the set of germs at the origin of such vector fields by $\mathcal{X}(\mathbb{R}^2,\mu)$.

We have the following:
\begin{lemma}
\label{lemma-bijection}
There exists a bijection between the space of univariate functions $C^{\infty}(\mathbb{R},\mathbb{R})$ and the space $\mathcal{X}(\mathbb{R}^2,\mu)$. If we denote by $X_f$ the vector field corresponding to function $f$, this bijection is given by:
$$f\mapsto X_f=-(1+x)f'(y)\frac{\partial}{\partial x}+f(y)\frac{\partial}{\partial x},\ X_f\mapsto \frac{1}{1+x}\mu (X_f).$$
\end{lemma}
\begin{remark}
Note that, if $X\in\mathcal{X}(\mathbb{R}^2,\mu)$, by setting $H=\mu(X)$, $X$ can be written in the familiar form $-\frac{\partial H}{\partial y}\frac{\partial}{\partial x}+\frac{\partial H}{\partial x}\frac{\partial}{\partial y}$. That is, vector fields of the plane, preserving the form $\mu$, are a special class of Hamiltonian fields.
\end{remark}
In the same way, we can obtain the form of diffeomorphisms preserving the $1$--form $\mu$, that is, diffeomorphisms $\phi$ such that $\phi^*\mu=\mu$. We denote their set by $Diff(\mathbb{R}^2,\mu)$.
\begin{lemma}
Let $\phi \in Diff(\mathbb{R}^2,\mu)$. Then:
$$\phi(x,y)=(\frac{1+x}{\psi '(y)}-1,\psi(y)),$$
where $\psi:(\mathbb{R},0)\rightarrow (\mathbb{R},0)$ is the germ at the origin of a smooth function with $\psi'(0)=1$. 
\end{lemma}
\begin{proof}
The form of $\phi$ is obtained by studying equation $\phi ^*\mu=\mu$, under the condition $\phi(0,0)=(0,0)$.
\end{proof}
We aim to classify members of $\mathcal{X}(\mathbb{R}^2,\mu)$, up to diffeomorphisms preserving $\mu$. If $X,Y\in \mathcal{X}(\mathbb{R}^2,\mu)$, we shall call them $\mu$--conjugate if a diffeomorphism $\phi \in Diff(\mathbb{R}^2,\mu)$ exists, such that $\phi _*Y=X$. We note that this equivalence relation induces an equivalence relation on the functions, on which the vector fields depend.
\begin{lemma}\label{vec-func}
Let $\psi:(\mathbb{R},0)\rightarrow (\mathbb{R},0),\ \psi'(0)=1$. The vector field $X\in \mathcal{X}(\mathbb{R}^2,\mu)$, corresponding to function $f$, is $\mu$--conjugate to the vector field $Y$, corresponding to the function $g(y)=\frac{1}{\psi'(y)}f(\psi(y))$.  
\end{lemma} 
\begin{proof}
It is $X=-(1+x)f'(y)\frac{\partial}{\partial x}+f(y)\frac{\partial}{\partial y}$ and 
$$Y=\bigg(\frac{(1+x)}{(\psi'(y))^2}\psi''(y)f(\psi(y))-(1+x)f'(\psi(y))\psi'(y)\frac{\partial}{\partial x}\bigg)+\frac{1}{\psi'(y)}f(\psi(y))\frac{\partial}{\partial y}.$$
Define $\phi(x,y)=(\frac{1+x}{\psi'(y)}-1,\psi(y))$. It is easy to confirm that $\phi \in Diff(\mathbb{R}^2,\mu)$ and $D\phi \cdot Y=X\circ \phi$.
\end{proof} 
Thus, to classify vector fields of the plane that preserve the form $\mu$, we shall concentrate on their dependence on univariate functions.
%=================================================================
\section{Local normal forms for members of $\mathcal{X}(\mathbb{R}^2,\mu)$}
According to \ref{vec-func}, functions, corresponding to $\mu$--conjugate vector fields, are related through a diffeomorphism tangent to the identity.

To make this precise, let $f,g$ be two function--germs, at the origin of $\mathbb{R}$. Relation $g=\frac{1}{\psi'}f(\psi)$, used in \ref{vec-func}, is nothing more than the smooth conjugacy relation for vector fields, in the $1$--dimensional case, where the conjugacy $\psi$ satisfies conditions $\psi(0)=0,\ \psi'(0)=1$. 

In the $C^{\infty}$ case, the classification of $1$--dimensional vector fields is well--known, see, for example, \cite{Belitskii1}. Of interest to us is the following result:
\begin{theorem}\label{Theo-capeiro-normal-forms}
Let $f(y)\frac{\partial}{\partial y}$ the germ at the origin of a smooth vector field of the line. This field is, locally, smoothly conjugate either with the vector field $\lambda y\frac{\partial}{\partial y}$, in case $f'(0)=\lambda \neq 0$, or with the vector field $(\pm y^k+dy^{2k-1})\frac{\partial}{\partial y},\ d\in \mathbb{R},k\geq 2$.  
\end{theorem}  
\begin{proof}
We shall state the main points of the proof of this theorem here, and refer the reader to \cite{Belitskii1} for details.

For the $C^{\infty}$ linearization, one just solves the conjugacy equation with respect to the unknown conjugacy $\psi$.

For the non--linear normal form, the proof consists of three steps: first, one shows that vector fields:
\[
(\pm y^k+dy^{2k-1})\frac{\partial}{\partial y}\hspace*{0.5cm} \text{and}\hspace*{0.5cm} (\pm y^k+dy^{2k-1}+x^{2k}g(y))\frac{\partial}{\partial y},
\]
where $g(y)$ is a $C^{\infty}$ local function, are conjugate. Then, one shows that the vector field $(ay^k+h(y))\frac{\partial}{\partial y}$, where $a\neq 0,h(0)=h'(0)=...=h^{r}(0)=0,k\geq 2$, is conjugate with $ay^k+dy^{2k-1}$. Third, a linear transformation is used to turn the coefficient of $y^k$ to $\pm 1$.  
\end{proof}
\begin{remark}
The sign of the term $y^k$, in the normal form $\pm y^k+dy^{2k-1}$ depends on $k$: if $k$ is odd, the sign is equal to "+", otherwise the sign is equal to the sign of $a$. In our studies, the conjugating diffeomorphism $\psi$ should satisfy $\psi'(0)=1$. The third step in the proof above is therefore not applicable. Thus, we shall use the normal form $ay^k+dy^{2k-1},a\neq 0$.
\end{remark}

Using \ref{Theo-capeiro-normal-forms}, one gets the following:
\begin{theorem}
Let $X\in \mathcal{X}(\mathbb{R}^2,\mu)$.
\begin{enumerate}
\item[(i)] {$X$ does not possess hyperbolic singularities.}
\item[(ii)] {If $X(0)\neq 0$, then $X$ is $\mu$--conjugate to either $X_0=a\frac{\partial}{\partial y}$, or to $X_1=-a(1+x)\frac{\partial}{\partial x}+ay\frac{\partial }{\partial y}$.}
\item[(iii)] {If $j^iX(0)=0,\ i=0,..,k-2$, and $j^{k-1}X(0)\neq 0,\ k\geq 2$, $X$ is $\mu$-conjugate to
\[X_k=-(1+x)(a ky^{k-1}+d(2k-1)y^{2k-2})\frac{\partial}{\partial x}+(a y^k+dy^{2k-1})\frac{\partial}{\partial y}.
\]}
\end{enumerate}
\end{theorem}
\begin{proof}
Vector field $X$ is of the form $X=-(1+x)f'(y)\frac{\partial}{\partial x}+f(y)\frac{\partial }{\partial y}.$
\begin{itemize}
\item[(i)] {Let $X(0,0)=(-f'(0),f(0))=(0,0)$. The jacobian matrix of $X$ at the origin is:
\[
\begin{bmatrix}
-f'(0) & -f''(0)\\
0 & f'(0)
\end{bmatrix}.
\]
Thus, its eigenvalues are both equal to $0$.}
\item[(ii)] {If $X(0,0)=(-f'(0),f(0))\neq (0,0)$, we distingulish two cases: either $f(0)\neq 0$ and $f'(0)=c\in \mathbb{R}$, or $f(0)=0,\ f'(0)\neq 0$.

In case $f(0)=a\neq 0$, let us consider the conjugacy equation:
\[
f(y)=\frac{1}{\psi '(y)}a\Rightarrow \psi'(y)=\frac{a}{f(y)}.
\]
Searching for a solution of the form $\psi(y)=y+\psi_1(y),\psi_1(0)=\psi_1'(0)=0$, we arrive at equation:
\[
\psi_1'(y)=\frac{a}{f(y)}-1,
\]
which has a $C^{\infty}$ right--hand side, thus possesses a smooth solution. Therefore, $f$ is smoothly conjugate, via a diffeomorphism of the form $\psi(y)=y+\psi_1(y)$, to the constant function $a$, while function $a$ is conjugate to function $b$, with $\psi$ of the desired form, if, and only if, $a=b$.
   
Thus, $f$ is conjugate to $a$, which corresponds to vector field $X_0$.

On the other hand, if $f(0)=0$ and $f'(0)=a\neq 0$, $f$ is conjugate to $ay$ (due to \ref{Theo-capeiro-normal-forms}), corresponding to the vector field $X_1$.}
\item[(iii)] {Let $j^iX(0)=0,\ i=0,..,k-2$ and $j^{k-1}X(0)\neq 0$. This means that $j^kf\neq 0$, thus $f$ is conjugate with the function $a y^k+dy^{2k-1}$, corresponding to the model $X_k$.}
\end{itemize}
\end{proof}
We summarize these results in Table $2$. Note that there are two kinds of regular points and no hyperbolic singularities.
\begin{center}
Table 2\\
Local models for members of $\mathcal{X}(\mathbb{R}^2,\mu)$
\end{center}    
{\small \begin{center}
   \begin{tabular}{lll}
   \hline
 type of singularity & local model \\
\hline 
regular point & $X_0=a\frac{\partial}{\partial y}$ \\ 
regular point & $X_1=-a(x+1)\frac{\partial}{\partial x}+ay\frac{\partial }{\partial y}$\\
hyperbolic singularity & $-$ \\
degenerate singularity & $X_k=-(1+x)(a ky^{k-1}+d(2k-1)y^{2k-2})\frac{\partial}{\partial x}+$\\
$\ $                  &$+(a y^k+dy^{2k-1})\frac{\partial}{\partial y}$

%  \hline 
       \end{tabular}
        \end{center}}
%=================================================================
\section{Bifurcations of members of $\mathcal{X}(\mathbb{R}^2,\mu)$}
To study bifurcations of vector fields of interest, we first recall, in our setting, the definition of a versal deformation.
\begin{definition}
Let $f$ be a univariate function (which corresponds to the vector field $f(y)\frac{\partial}{\partial y}$). A deformation of $f$ is a mapping $F:U\times \mathbb{R}\rightarrow \mathbb{R}$, where $U\subset \mathbb{R}^m$ an open neighbourhood of the origin, such that $F(0,y)=f(y)$. This deformation is said to be versal, if for every other deformation $G:V\times \mathbb{R}\rightarrow \mathbb{R}$, where $V\subset \mathbb{R}^l$ an open neighbourhood of the origin, there exist germs $\psi:V\times \mathbb{R}\rightarrow \mathbb{R}$ and $H:V\rightarrow U$, such that:
\begin{itemize}
\item[i)] {$\forall \lambda_0 \in V$, $\psi(\lambda_0,z)$ is the germ at the origin of a diffeomorphism of the line, such that $\psi(\lambda_0,0)=0,\psi'(\lambda_0,0)=1$}
\item[ii)] {$G(\lambda,z)=(\frac{\partial \psi(\lambda,z)}{\partial z})^{-1}F(H(\lambda),\psi(\lambda,z))$.}
\end{itemize}   
\end{definition}
Versal deformations of vector fields of the line have been studied in \cite{Kostov} (see also \cite{Klimes-Rousseau}). Here we are interested in the following result:
\begin{theorem}\label{theo-Kostov}
A versal unfolding of the vector field $\Big(a y^k+\lambda_k y^{2k-1}\Big)\frac{\partial}{\partial y}$ is given by $F:\mathbb{R}^k\times \mathbb{R}\rightarrow \mathbb{R}$, defined as:
\[\ F(\lambda_1,..,\lambda_k,y)=\Big(a y^k+\sum_{i=1}^{k-1}\lambda_iy^{k-1-i}+\lambda_ky^{2k-1}\Big)\frac{\partial}{\partial y}.
\]
\end{theorem}

For a proof of the theorem above, we refer the reader to \cite{Kostov,Klimes-Rousseau}. This theorem, along with the following proposition, will allow us to describe bifurcations for members of $\mathcal{X}(\mathbb{R}^2,\mu)$.
\begin{theorem}\label{theo-unfoldings}
The unfolding:
\begin{equation}
\begin{split}
F(\lambda,x,y)=\ &-(1+x)\Big[a ky^{k-1}+\sum_{i=1}^{k-1}\Big((k-1-i)\lambda_iy^{k-2-i}\Big)+\\
&+(2k-1)\lambda_{k}y^{2k-2}\Big]\frac{\partial}{\partial x}+ \\
\ &+(a y^k+\sum_{i=1}^{k-1}\lambda_iy^{k-1-i}+\lambda_k y^{2k-1})\frac{\partial}{\partial y},
\end{split}
\end{equation}
is a versal unfolding for the vector field:
\[
X_k=-(1+x)(a ky^{k-1}+\lambda_k(2k-1)y^{2k-2})\frac{\partial}{\partial x}+(a y^k+\lambda_ky^{2k-1})\frac{\partial}{\partial y}.
\]
\end{theorem}
\begin{proof}
Consider mapping:
\[
X_f\mapsto \frac{1}{1+x}\mu(X_f)\ \ (\star)
\]
introduced in \ref{lemma-bijection}. This mapping is continuous, thus the preimage of an open neighbourhood of $f$ is an open neighbourhood of $X_f$.

Since the image of a versal unfolding $F_f$ of $f(y)\frac{\partial}{\partial y}$ is an open neighbourhood of $f(y)\frac{\partial}{\partial y}$, the preimage of this neighbourhood, under mapping $(\star)$, is a versal unfolding of $X_f$. Versal unfoldings for vector fields of the line are given in \ref{theo-Kostov}; using the inverse of mapping ($\star$) we get the conclusion.
\end{proof}
We are now able to state results regarding bifurcations of vector fields belonging in $\mathcal{X}(\mathbb{R}^2,\mu)$. Since vector fields undergo no bifurcations in the neighbourhood of regular points and, as we saw above, hyperbolic singularities do not exist, we proceed to the analysis of the first degenerate case.
\begin{proposition}
There exists a codimension $2$ surface in $\mathcal{X}(\mathbb{R}^2,\mu)$, every member of which is $\mu$--conjugate to the vector field:
\[X_2=-(1+x)(2ay+3\lambda_2y^2)\frac{\partial}{\partial x}+(ay^2+\lambda_2y^3)\frac{\partial}{\partial y}.\]
A versal unfolding of this field is given by:
\[F_2=-(1+x)\Big[2ay+3\lambda_2y^2\Big]\frac{\partial}{\partial x}+(ay^2+\lambda_1+\lambda_2y^3)\frac{\partial }{\partial y}.\]
\end{proposition}
\begin{proof}
The proof is, of course, a consequence of \ref{theo-unfoldings}. The codimension of the surface is equal to the number of the parameters of the unfolding. 
\end{proof} 

In figure 1 we see the face portrait of the vector field $X_2$, for characteristic values of parameters $a,\lambda_2$. Apart from the x--axis, which consists entirely of singularities, there exists also a saddle point (for $a\neq 0$), which lies on the $x=-1$ line, on which the Martinet form descents to the form $\pm zdz$.
\begin{figure}[h]
\begin{center}
\includegraphics[scale=0.4]{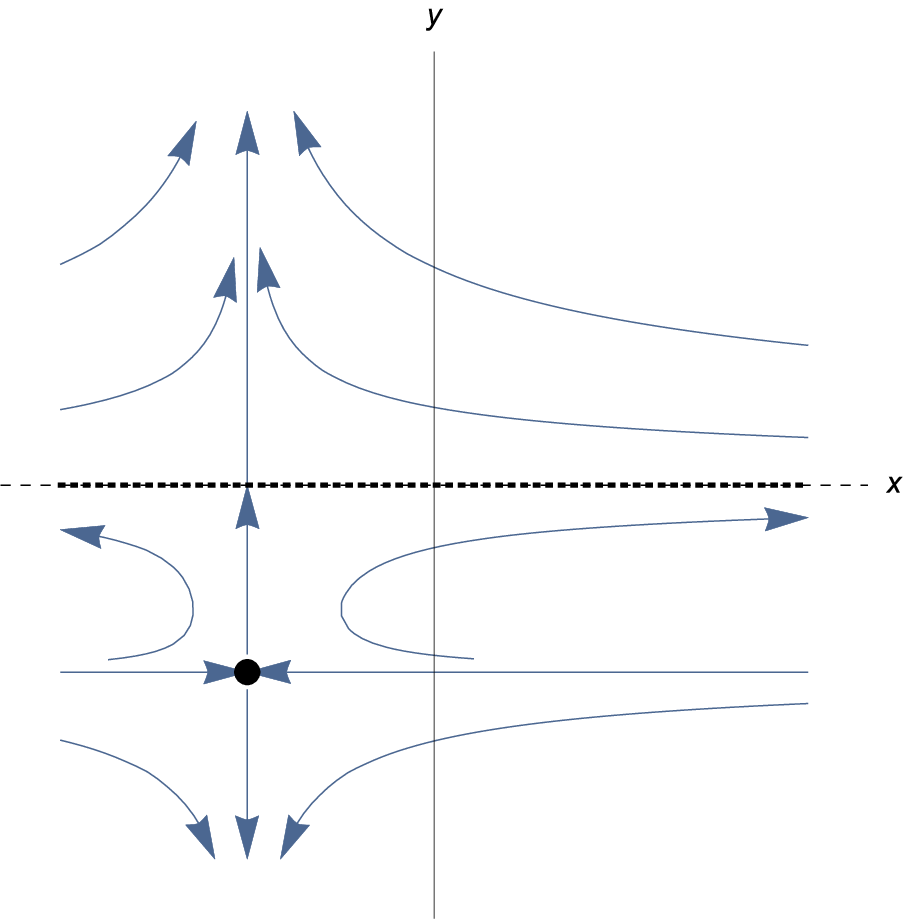}
\includegraphics[scale=0.4]{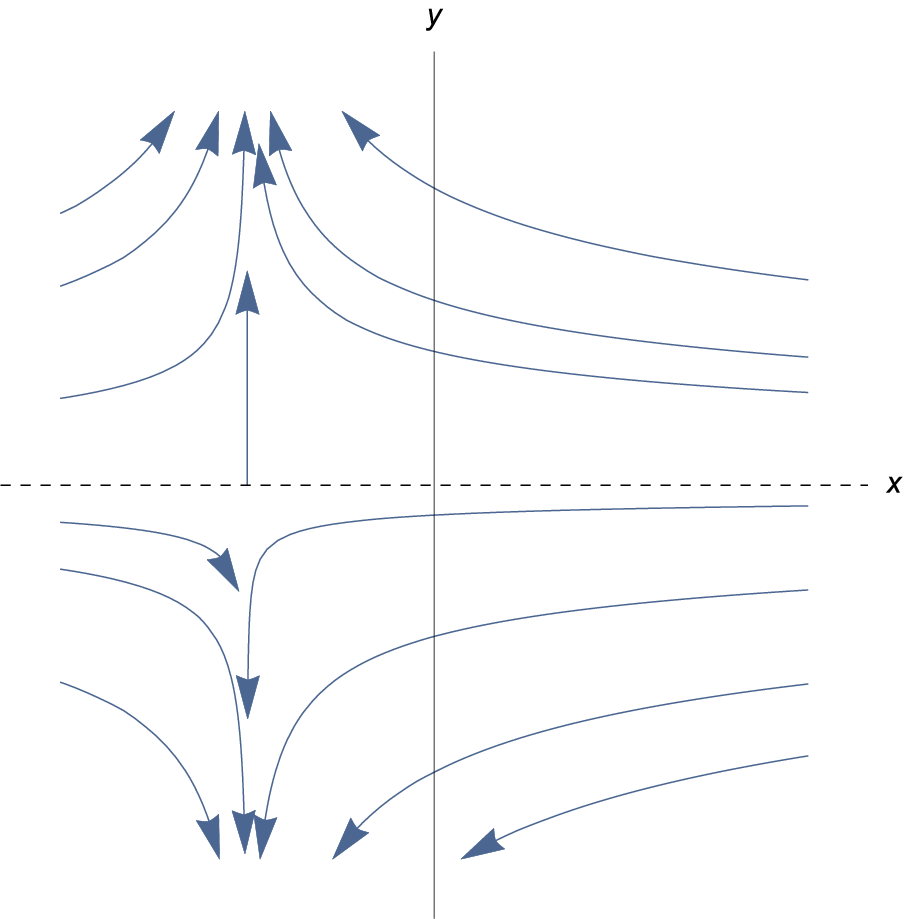}
\includegraphics[scale=0.4]{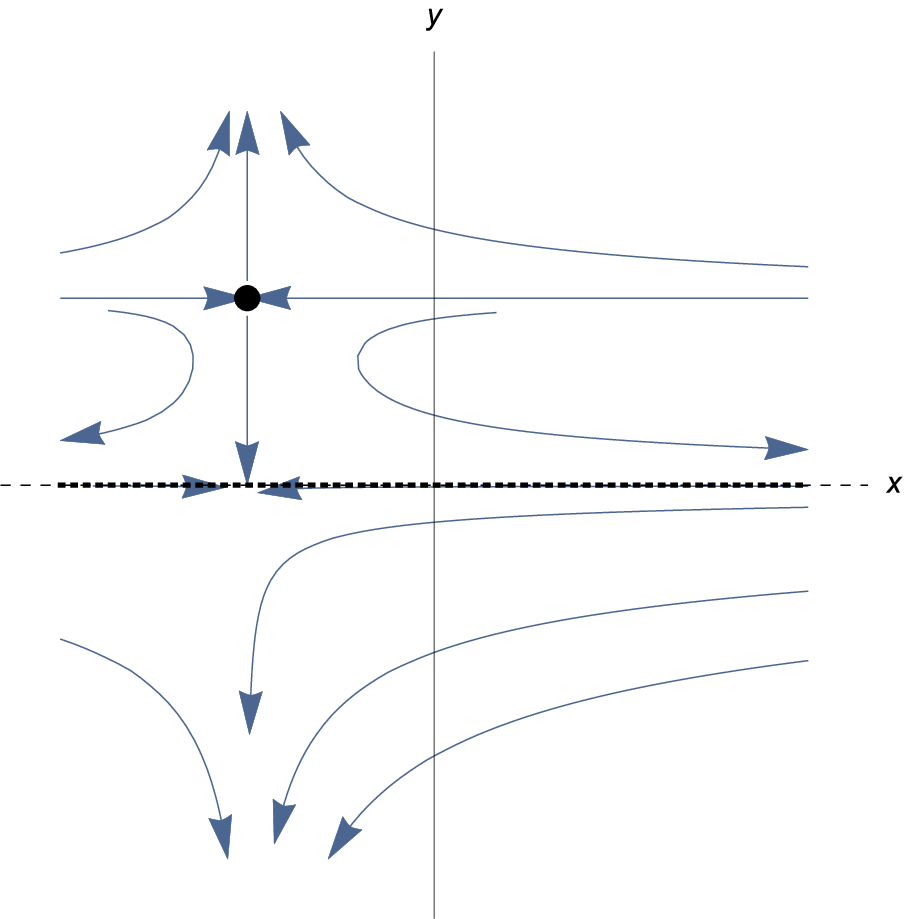}
\end{center}
\caption{Phase portraits for the $X_2$ vector field. Parameter values: $a=1,\lambda_2=1$ (left), $a=0,\lambda_2=1$ (center), $a=-1,\lambda_2=1$ (right). The dotted x--axis consists entirely of fixed points. As $a$ decreases from positive to negative values, a saddle singularity merges with the fixed points of the x--axis and then reappears.}
\end{figure}

For the unfolding $F_2$, the phase portrait is entirely different. For $\lambda_1\neq 0$, the origin is no longer an equilibrium. On the invariant $x=-1$ line there may no exist one or three equilibria, according to the parameter values. Characteristic phase portraits are shown in figure 2.
\begin{figure}[h]
\begin{center}
\includegraphics[scale=0.5]{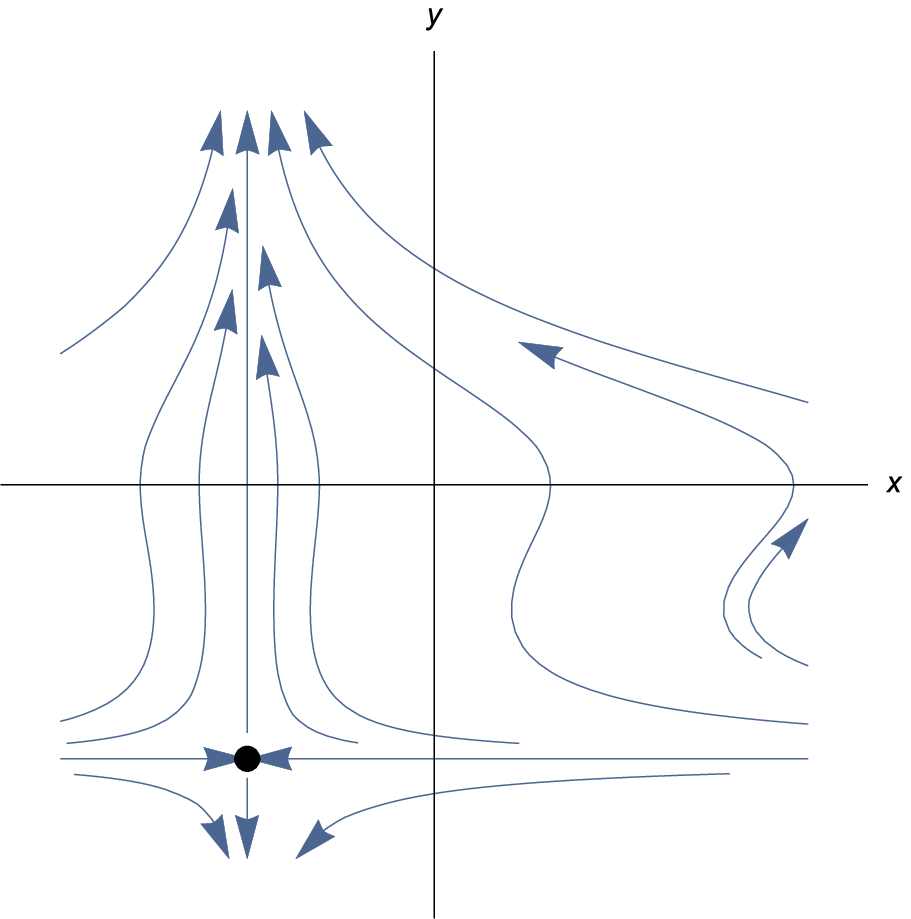}
\includegraphics[scale=0.5]{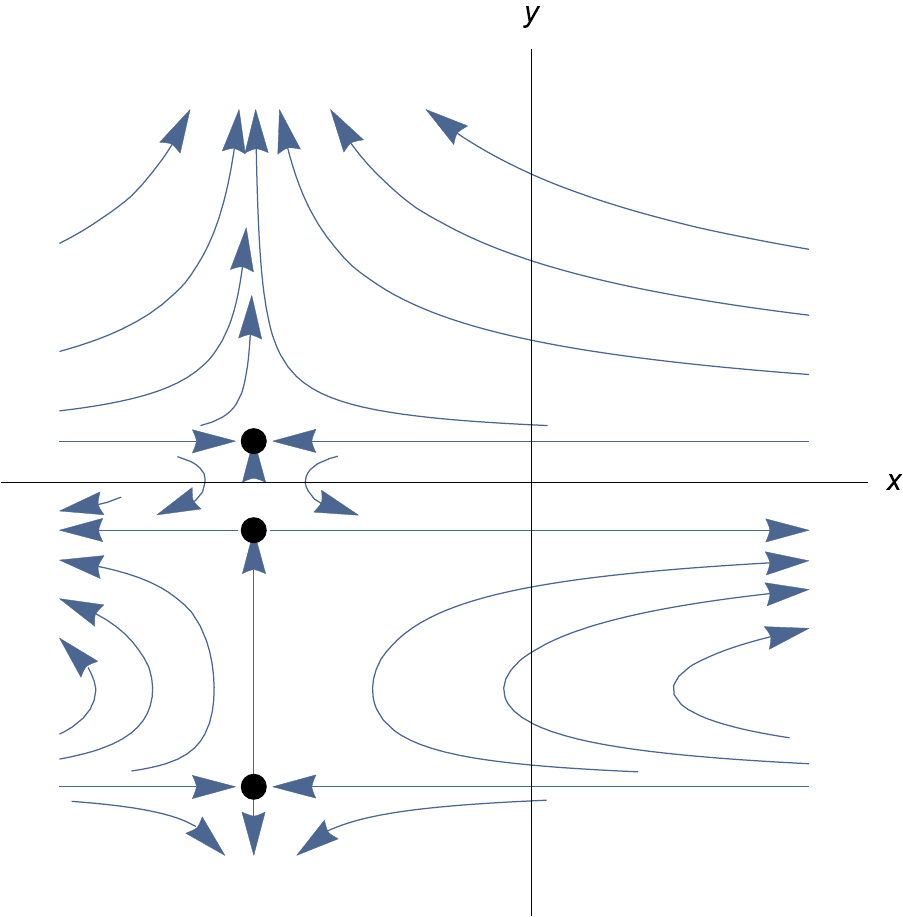}
\end{center}
\caption{Phase portraits for the $F_2$ vector field. Parameter values: $a=\lambda_1=\lambda_2=1$ (left) and $a=\lambda_2=1,\ \lambda_1=-0.02$ (right). Equilibria exist only on the, invariant, $x=-1$ line.}
\end{figure} 

Analogously, one can study all fields listed in \ref{theo-unfoldings}.
 %===============================================================
 \section{Conclusions}
 We have studied vector fields of $\mathbb{R}^3$ which preserve the Martinet $1$--form $\alpha$. Their third component vanishes identically; thus we turned our attention to vector fields of the plane preserving the form $\mu$. We gave a complete list of their local normal forms and studied their bifurcations, by constructing versal unfoldings of these normal forms. These results are connected with other studies, concerning vector fields preserving special structures (\cite{Kourliouros1,Kourliouros2,Anastassiou}).
 
An analogous classification for diffeomorphisms of the plane preserving the form $\mu$, and, thus, of $3$--d diffeomorphisms preserving the form of Martinet, cannot be given, since, as implied by Lemma $2$, such a diffeomorphism is not finitely determined (except, of course, from the identity transformation).

The local, and global, study of dynamical systems preserving specific $1$--forms is, of course, far from complete. In a next publication, we hope to further comment on these issues.     

%================================================================
\section*{Conflict of interest}
The authors declare that they have no
conflicts of interest.
%================================================================
%\section*{References}

\end{document}